\documentclass[preprint,12pt]{elsarticle}\usepackage{amsmath,amssymb,amsthm,mathtools} 
\usepackage{enumitem} 
\usepackage{hyperref} 
\title{Monotone maximum partial-twuality widths of vf-safe delta-matroids}
\author{Qi Yan, Zhao Zhao\\
\small School of Mathematics and Statistics, Lanzhou University, PR China\\
\small{\tt Email: yanq@lzu.edu.cn; zhzhao2025@lzu.edu.cn }}
\date{} 

\newtheorem{theorem}{Theorem}[section] 
\newtheorem{corollary}[theorem]{Corollary} 
\newtheorem{proposition}[theorem]{Proposition} 
\newtheorem{lemma}[theorem]{Lemma} 
\newtheorem{example}[theorem]{Example} 
 
\newtheorem{definition}[theorem]{Definition}

\begin{document}

\begin{abstract} 
For a delta-matroid, the maximum twist width theorem states that the maximum
width over all twists can be reached along a non-decreasing sequence of
intermediate twist widths. In this paper we study analogous monotone maximum
width sequences for partial twualities generated by twist and loop complementation.
We prove that, for each non-twist partial-twuality operation on a vf-safe
delta-matroid, there exists a subset attaining the corresponding maximum
partial-twuality width whose elements can be ordered so that the successive
intermediate widths are non-decreasing. Together with the known twist case, this
gives a monotone maximum width theorem for all five nontrivial partial-twuality
operations on vf-safe delta-matroids. We also prove feasible-set attainment results
for the operations \(\ast\times\) and \(\ast\times\ast\).
Finally, we translate these results to ribbon graphs, obtaining monotone sequences
for maximum partial-twuality Euler genera and spanning quasi-tree attainment
results for the corresponding ribbon graph operations.
\end{abstract} 
\maketitle 
\section{Introduction} 
Partial duality of ribbon graphs was introduced by Chmutov \cite{Chmutov2009} as a far-reaching generalization of geometric duality. It has since become an important operation in topological graph theory and in the study of graph polynomials on surfaces. One natural problem is to understand how the genus, or Euler genus, of a ribbon graph changes under partial duality. 

Recently, Chen, Gross and Tucker \cite{ChenGrossTucker2026}  obtained formulas for the maximum partial-dual genus of orientable ribbon graphs and for the maximum partial-dual Euler genus of arbitrary ribbon graphs. A key point of their work is that this maximum can be realized by taking the partial dual with respect to the edge set of a spanning quasi-tree. They also posed a monotonicity problem: whether one can reach the maximum partial-dual Euler genus by dualizing edges one at a time, so that the intermediate Euler genera never decrease. 

Delta-matroids provide a natural matroid framework for this problem. If \(G\) is a ribbon graph, then its associated delta-matroid \(D(G)\) has as feasible sets precisely the edge sets of spanning quasi-trees of \(G\) \cite{ChunMoffattNobleRueckriemen2019JCTA}. Moreover, partial duality of ribbon graphs corresponds to twist of delta-matroids, and the width of \(D(G)\) is equal to the Euler genus of \(G\) \cite{ChunMoffattNobleRueckriemen2019JCTA}. Thus the maximum partial-dual Euler genus problem for ribbon graphs has a natural delta-matroid counterpart: the maximum twist width problem. 

This point of view was developed by Jin, Li, Yan and Zhang \cite{JinLiYanZhang2026}. They showed that the maximum twist width of a delta-matroid can be attained by twisting a feasible set, thereby extending the spanning quasi-tree attainment result from ribbon graphs to delta-matroids. They also proved the corresponding monotonicity theorem: the feasible set attaining the maximum twist width can be ordered so that the intermediate twist widths form a non-decreasing sequence. In particular, this gives an affirmative answer to the monotonicity problem of Chen, Gross and Tucker for ribbon graphs. 

The purpose of this paper is to continue this line of research from partial duality to partial twuality. Besides twist, there is another basic operation on set systems, namely loop complementation. Twist and loop complementation generate an \(S_3\) action on each element of a set system \cite{BrijderHoogeboom2011}. The five nontrivial operations may be represented by \( \{\ast, \times, \ast\times, \times\ast, \ast\times\ast\}. \) From the viewpoint of ribbon graphs, these operations correspond to partial duality,
partial Petriality, and their compositions \cite{ChunMoffattNobleRueckriemen2019PLMS}. It is therefore natural to ask whether the monotone maximum width phenomenon for twists extends to the other partial-twuality operations. There is, however, an essential difference between twists and the operations involving loop complementation. Loop complementation does not preserve arbitrary delta-matroids. For this reason, the appropriate setting is the class of vf-safe delta-matroids, namely delta-matroids that remain delta-matroids under every sequence of twists and loop complementations. This class includes both binary delta-matroids and ribbon-graphic delta-matroids \cite{ChunMoffattNobleRueckriemen2019PLMS}. 

Our main result proves that the monotone maximum width conclusion holds for the two operations \(\ast\times\) and \(\times\ast\) on vf-safe delta-matroids. More precisely, for each of these operations, there exists a subset attaining the maximum partial-twuality width, and its elements can be ordered so that the corresponding intermediate widths are non-decreasing. We also record the corresponding results for the remaining operations. The case of \(\times\) follows from a simpler version of the same argument, and the case of \(\ast\times\ast\) follows from the \(\times\)-case by duality. Together with the known twist case, this shows that all five nontrivial partial-twuality operations admit monotone maximum width sequences in the vf-safe setting. 

Finally, we translate the results back to ribbon graphs. Since ribbon-graphic delta-matroids are vf-safe and width corresponds to Euler genus, our theorems yield monotone sequences for maximum Euler genera under the corresponding partial twualities of ribbon graphs. 

\section{Preliminaries} 

\subsection{Delta-matroids} 

Let \(E\) be a finite set. A \emph{set system} on \(E\) is a pair \(D=(E,\mathcal F)\), where \(\mathcal F\subseteq 2^E\). The members of \(\mathcal F\) are called \emph{feasible sets}. The set system is \emph{proper} if \(\mathcal F\neq\emptyset\). For \(X,Y\subseteq E\), write \[ X\triangle Y=(X\cup Y)\setminus(X\cap Y) \] for their symmetric difference. 

\begin{definition}[\cite{Bouchet1987}]
A proper set system \(D=(E,\mathcal F)\) is a delta-matroid if it satisfies the symmetric exchange axiom: for all \(X,Y\in\mathcal F\) and every \(u\in X\triangle Y\), there exists \(v\in X\triangle Y\) (possibly \(v=u\)) such that \[ X\triangle\{u,v\}\in\mathcal F. \] 
\end{definition} 

If all feasible sets have the same cardinality, then \(D\) is a \emph{matroid} in the sense that \(\mathcal F\) is the set of bases of a matroid on \(E\). For a proper set system \(D=(E,\mathcal F)\), 
let $$D_{\max} = (E, \mathcal{F}_{\max}(D))$$ and $$D_{\min} = (E, \mathcal{F}_{\min}(D)).$$
Let $r(D_{\max})$ and $r(D_{\min})$ denote the sizes of the largest and smallest feasible sets of $D$, respectively. 
The \emph{width} of $D$, denoted by $\omega(D)$, is defined by
\[
\omega(D) = r(D_{\max}) - r(D_{\min}).
\]
For $0\leq i\leq \omega(D)$, define
\[
\mathcal F_{\min+i}(D)=\{F\in\mathcal F\mid |F|=r(D_{\min})+i\},
\]
and
\[
\mathcal F_{\max-i}(D)=\{F\in\mathcal F\mid |F|=r(D_{\max})-i\}.
\]

\subsection{Twists and loop complementations} 

\begin{definition}[\cite{Bouchet1987}] 
Let \(D=(E,\mathcal F)\) be a set system and let \(A\subseteq E\). The twist of \(D\) with respect to \(A\) is the set system \[ D^{\ast| A }= (E,\{F\triangle A \mid F\in\mathcal F\}). \] The dual of \(D\) is \(D^\ast=D^{\ast| E}.\) \end{definition} 

\begin{definition}[\cite{BrijderHoogeboom2011}]
Let \(D=(E,\mathcal F)\) be a set system and let \(e\in E\). The loop complementation of \(D\) on \(e\), denoted \(D^{\times |e}\), is the set system \(D^{\times| e}=(E,\mathcal F')\), where 
\[ \mathcal F' = \mathcal F \triangle \{F\cup\{e\}\mid F\in\mathcal F,\ e\notin F\}. \] 
For \(A\subseteq E\), define \(D^{\times| A}\) by applying loop complementation to every element of \(A\). Loop complementations on distinct elements commute, so \(D^{\times |A}\) is well-defined. \end{definition} 

Twist and loop complementation on a single element generate a group isomorphic to
\(S_3\). We use the convention that words are read from left to right. Thus, for a
word \(\sigma\) in \(\ast\) and \(\times\), and for \(A\subseteq E\), the notation
\(D^{\sigma|A}\) means that the operation \(\sigma\) is applied to every element
of \(A\). Operations on distinct elements commute, so this notation is
unambiguous. For example,
\[
        D^{\ast\times|A}
        =
        (D^{\ast|A})^{\times|A},
\]
and
\[
        D^{\times\ast|A}
        =
        (D^{\times|A})^{\ast|A}.
\]

A delta-matroid \(D\) is \emph{vf-safe} \cite{ChunMoffattNobleRueckriemen2019PLMS} if every set system obtained from \(D\) by an arbitrary sequence of twists and loop complementations is again a delta-matroid.  The vf-safe hypothesis is therefore natural, since loop complementation does not preserve arbitrary delta-matroids.

\subsection{ Element types}
We next recall the element types used in the single element width change table. Let \(D=(E,\mathcal F)\) be a proper set system. An element \(e\in E\) is called a \emph{loop} of \(D\) if \(e\) is contained in no feasible set of \(D\).

\begin{definition}[\cite{ChunMoffattNobleRueckriemen2019PLMS}]

Let \(D=(E,\mathcal F)\) be a proper set system, and let \(e\in E\). \begin{description} \item[\textup{(1)}] The element \(e\) is called a {ribbon loop} of \(D\) if \(e\) is a loop of \(D_{\min}\).

\item[\textup{(2)}] A ribbon loop \(e\) is called {non-orientable} if \(e\) remains a ribbon loop in \(D^{\ast|e}\). Otherwise, the ribbon loop \(e\) is called {orientable}. \end{description}
\end{definition}

Based on the above definition, each element within a set system can be categorized into three distinct primal types. For a set system $D = (E, \mathcal{F})$ and an element $e \in E$, the primal type of $e$ is classified as $p$, $u$, or $t$. Specifically, $e$ is of type $p$ if it is not a ribbon loop, of type $u$ if it constitutes an orientable ribbon loop, and of type $t$ if it corresponds to a non-orientable ribbon loop.
The dual type of $e$ in $D$ refers to the primal type of the same element $e$ evaluated in the dual set system $D^{*}$. The overall type of an element is defined as a concatenation of its primal type and dual type, with the primal type placed in the first position and the dual type in the second. For illustration, an element with type $pu$ possesses primal type $p$ and dual type $u$.

The following proposition from \cite{YanJin2024} gives equivalent characterizations
of primal and dual element types in set systems.

\begin{proposition}[\cite{YanJin2024}]
\label{lemma6}
For a set system $D=(E,\mathcal{F})$ and $e\in E$, the following statements hold.
\begin{description}
\item[\textup{(1)}] The primal type of $e$ is $p$ in $D$ if and only if there exists $F \in \mathcal{F}_{\mathrm{min}}(D)$ such that $e \in F$.
\item[\textup{(2)}] The primal type of $e$ is $u$ in $D$ if and only if for every $F \in \mathcal{F}_{\mathrm{min}}(D) \cup \mathcal{F}_{\mathrm{min}+1}(D)$, $e \notin F$.
\item[\textup{(3)}] The primal type of $e$ is $t$ in $D$ if and only if for every $F \in \mathcal{F}_{\mathrm{min}}(D)$, $e \notin F$, and there exists $F_1 \in \mathcal{F}_{\mathrm{min}+1}(D)$ such that $e \in F_1$.

\item[\textup{(4)}] The dual type of $e$ is $p$ in $D$ if and only if there exists $F \in \mathcal{F}_{\mathrm{max}}(D)$ such that $e \notin F$.
\item[\textup{(5)}] The dual type of $e$ is $u$ in $D$ if and only if for every $F \in \mathcal{F}_{\mathrm{max}}(D) \cup \mathcal{F}_{\mathrm{max}-1}(D)$, $e \in F$.
\item[\textup{(6)}] The dual type of $e$ is $t$ in $D$ if and only if for every $F \in \mathcal{F}_{\mathrm{max}}(D)$, $e \in F$, and there exists $F_1 \in \mathcal{F}_{\mathrm{max}-1}(D)$ such that $e \notin F_1$.
\end{description}
\end{proposition}

\begin{table}[t]
  \centering
  \caption{The difference of $\omega(D^{\bullet|e})-\omega(D)$ for any $\bullet \in \{\ast,\times,\ast\times,\times\ast,\ast\times\ast\}$.}
  \label{tab:widths}
  \begin{tabular}{lccccc}
\hline
Type of $e$ & $\ast$ & $\times$ & $\ast\times$ & $\times\ast$ & $\ast\times\ast$ \\
\hline
$pp$ & $+2$ & $+1$ & $+2$ & $+2$ &$+1$\\
$uu$ & $-2$ & $0$ & $-1$ & $-1$ &$0$\\
$pu$ & $0$ & $0$ & $+1$ & $0$ &$+1$\\
$up$ & $0$ & $+1$ & $0$ & $+1$ &$0$\\
$tp$ & $+1$ & $+1$ & $+1$ & $0$ &$-1$\\
$tu$ & $-1$ & $0$ & $0$ & $-2$ &$-1 $\\
$pt$ & $+1$ & $-1$ & $0$ & $+1$ &$+1$\\
$ut$ & $-1$ & $-1$ & $-2$ & $0$ &$0$\\
$tt$ & $0$ & $-1$ & $-1$ & $-1$ &$-1$\\
\hline
\end{tabular}
\end{table}

\section{Feasible-set attainment for \texorpdfstring{\(\ast\times\)}{*x} and
\texorpdfstring{\(\ast\times\ast\)}{*x*}}
\label{sec:feasible-attainment}

For the twist operation, Jin et al.~\cite{JinLiYanZhang2026} proved that the maximum twist width can be attained by twisting a feasible set. Motivated by this result, we ask whether analogous feasible-set attainment holds for other partial-twuality operations. In this section, we show that it holds for \(\ast\times\) and \(\ast\times\ast\), but fails in general for \(\times\) and \(\times\ast\), even within the class of vf-safe delta-matroids.

\begin{theorem}
\label{thm:star-cross-feasible-max}
Let \(D=(E,\mathcal F)\) be a proper set system. Then there exists a feasible set
\(F\in\mathcal F\) such that
\[
        \omega(D^{\ast\times|F})
        =
        \max_{B\subseteq E}\omega(D^{\ast\times|B}).
\]
\end{theorem}

\begin{proof}
It suffices to show that for every \(A\subseteq E\), there exists a feasible set
\(F\in\mathcal F\) such that
\[
        \omega(D^{\ast\times|F})
        \geq
        \omega(D^{\ast\times|A}).
\]

Fix \(A\subseteq E\). Choose a feasible set \(F\in\mathcal F\) such that
\(|F\triangle A|\)
is as small as possible among all feasible sets of \(D\). We shall construct a sequence of subsets
\[
        A=A_0,A_1,\ldots,A_m=F
\]
by moving from \(A\) to \(F\) one element at a time, and we will show that the
\(\ast\times\)-width does not decrease along this sequence. 

Suppose that \(A_i\neq F\). Choose \(f\in A_i\triangle F\), and set
\[
        A_{i+1}=A_i\triangle\{f\}.
\]
Then \(A_{i+1}\) differs from \(A_i\) only in the element \(f\), and
\[
        |A_{i+1}\triangle F|=|A_i\triangle F|-1.
\]

We first observe that \(F\) is still a closest feasible set to \(A_i\). By construction, \(A\triangle A_i\) and \(A_i\triangle F\) form a disjoint partition
of \(A\triangle F\). Therefore
\[
        |A\triangle F|
        =
        |A\triangle A_i|+|A_i\triangle F|.
\]
If some feasible set \(F'\) satisfied
\[
        |F'\triangle A_i|<|F\triangle A_i|,
\]
then 
\[
        |F'\triangle A|
        \leq
        |F'\triangle A_i|+|A_i\triangle A|
        <
        |F\triangle A_i|+|A_i\triangle A|
        =
        |F\triangle A|,
\]
contradicting the choice of \(F\).

For the current \(i\), set
\( H_i=D^{\ast\times|A_i}.\)
We claim that \(f\) has primal type \(p\) in \(H_i\). Indeed, the feasible sets of \(D^{\ast|A_i}\) are precisely
\(\{M\triangle A_i\mid M\in\mathcal F\}\).
Since \(F\) is closest to \(A_i\), the set $F\triangle A_i$
is a minimum feasible set of \(D^{\ast| A_i}\).
We now apply loop complementation on \(A_i\). By the definition of loop
complementation, minimum feasible sets are preserved. Hence
\[
        F\triangle A_i\in
        \mathcal F_{\min}\bigl((D^{\ast|A_i})^{\times|A_i}\bigr).
\]
Therefore \(F\triangle A_i\) is a minimum feasible set of
\[H_i=D^{\ast\times|A_i}=(D^{\ast|A_i})^{\times|A_i}.\]
Since \(f\in F\triangle A_i\), there exists a minimum feasible set of \(H_i\)
containing \(f\). Hence, by Lemma \ref{lemma6} , the primal type of \(f\) in \(H_i\) is \(p\).
By Table \ref{tab:widths}, if an element
has primal type \(p\), then applying either \(\ast\times\) or \(\times\ast\) to that
element does not decrease the width.
There are two cases.

If \(f\in F\setminus A_i\), then $A_{i+1}=A_i\cup\{f\}.$
Since \(f\notin A_i\), applying \(\ast\times\) to \(f\) adds \(f\) to the partial
\(\ast\times\)-set. Hence $D^{\ast\times|A_{i+1}} =H_i^{\ast\times|f}.$
Therefore
\[
        \omega(D^{\ast\times|A_{i+1}})
        =
        \omega(H_i^{\ast\times|f})
        \geq
        \omega(H_i)
        =
        \omega(D^{\ast\times|A_i}).
\]

If \(f\in A_i\setminus F\), then $ A_{i+1}=A_i\setminus\{f\}.$
Since \(f\in A_i\), the operation \(\ast\times\) has already been applied to \(f\).
To remove it, we apply the inverse operation. Because $(\ast\times)^{-1}=\times\ast,$
we have $ D^{\ast\times|A_{i+1}}
        =
        H_i^{\times\ast|f}.$
Hence
\[
        \omega(D^{\ast\times|A_{i+1}})
        =
        \omega(H_i^{\times\ast|f})
        \geq
        \omega(H_i)
        =
        \omega(D^{\ast\times|A_i}).
\]

Thus in either case, replacing \(A_i\) by \(A_{i+1}\) moves one step closer to \(F\)
and does not decrease the width. After finitely many steps, we reach
       $ A_m=F.$
Therefore
\[
        \omega(D^{\ast\times|F})
        =
        \omega(D^{\ast\times|A_m})
        \geq
        \omega(D^{\ast\times|A_0})
        =
        \omega(D^{\ast\times|A}).
\]
Since \(A\subseteq E\) was arbitrary, the theorem follows.
\end{proof}

\begin{lemma}[\cite{BrijderHoogeboom2011}]
\label{lem:star-cross-star-parity}
Let \(D=(E,\mathcal F)\) be a proper set system and let \(A,Z\subseteq E\). Then
\(Z\in\mathcal F(D^{\ast\times\ast|A})\) if and only if
\[
        \left|
        \{X\in\mathcal F(D)\mid Z\subseteq X\subseteq Z\cup A\}
        \right|
\] is odd.
\end{lemma}

\begin{theorem} \label{thm:star-cross-star-feasible-max} Let \(D=(E,\mathcal F)\) be a proper set system. Then there exists a feasible set \(F\in\mathcal F\) such that \[ \omega(D^{\ast\times\ast|F}) = \max_{B\subseteq E}\omega(D^{\ast\times\ast|B}). \] 
\end{theorem} 
\begin{proof} We first note that \[ r({D^{\ast\times\ast|A}}_{\max})=r(D_{\max}) \] for every \(A\subseteq E\). Indeed, by Lemma~\ref{lem:star-cross-star-parity}, if \(Z\in\mathcal F(D^{\ast\times\ast|A})\), then there exists \(X\in\mathcal F(D)\) such that \(Z\subseteq X\). Hence \(|Z|\leq r(D_{\max})\). Therefore \[ r({D^{\ast\times\ast|A}}_{\max})\leq r(D_{\max}). \] Conversely, if \(Z\in\mathcal F_{\max}(D)\), then \[ \{X\in\mathcal F(D)\mid Z\subseteq X\subseteq Z\cup A\} \] contains only \(Z\). Hence its cardinality is odd. By Lemma~\ref{lem:star-cross-star-parity},  \(Z\in\mathcal F(D^{\ast\times\ast|A})\). 
Thus  \[ r({D^{\ast\times\ast|A}}_{\max})\geq r(D_{\max}). \] Therefore \[ r({D^{\ast\times\ast|A}}_{\max})=r(D_{\max}).\]

Now fix \(A\subseteq E\), and choose \( Z\in\mathcal F_{\min}(D^{\ast\times\ast|A}). \) By Lemma~\ref{lem:star-cross-star-parity}, the set \[ \mathcal I = \{X\in\mathcal F(D)\mid Z\subseteq X\subseteq Z\cup A\} \] has odd cardinality. Choose \(F\in\mathcal I\) such that there is no \(F'\in\mathcal I\) with
\(F'\subsetneq F\). Then \(F\in\mathcal F(D)\). We claim that \(Z\in\mathcal F(D^{\ast\times\ast|F})\). Indeed, applying Lemma~\ref{lem:star-cross-star-parity} with \(F\) in place of \(A\), we need to consider \[ \{X\in\mathcal F(D)\mid Z\subseteq X\subseteq Z\cup F\}. \] Since \(Z\subseteq F\), we have \[ \{X\in\mathcal F(D)\mid Z\subseteq X\subseteq F\}. \] By the inclusion-minimality of \(F\), this set contains only \(F\). Hence it has odd cardinality, and so \( Z\in\mathcal F(D^{\ast\times\ast|F}). \) Therefore \[ r({D^{\ast\times\ast|F}}_{\min}) \leq |Z| = r({D^{\ast\times\ast|A}}_{\min}). \] 
Together with \[ r({D^{\ast\times\ast|F}}_{\max}) = r(D_{\max}) = r({D^{\ast\times\ast|A}}_{\max}), \] this gives \[ \omega(D^{\ast\times\ast|F}) \geq \omega(D^{\ast\times\ast|A}). \] Since \(A\subseteq E\) was arbitrary, the maximum of \(\omega(D^{\ast\times\ast|B})\) over all \(B\subseteq E\) is attained at some feasible set \(F\in\mathcal F\).
\end{proof}

The analogous feasible-set attainment property does not hold for \(\times\) or
\(\times\ast\).

\begin{example}
\label{ex:cross-and-cross-star-no-feasible-attainment}
Let \(D=(\{e\},\{\emptyset\}).\)
Then \(D\) is a normal binary delta-matroid, and hence is vf-safe. Moreover,
\(\omega(D)=0. \)
For both \(\sigma=\times\) and \(\sigma=\times\ast\), we have
\[
        D^{\sigma|e}=(\{e\},\{\emptyset,\{e\}\}),
\]
and hence
\[
        \omega(D^{\sigma|e})=1.
\]
Therefore
\[
        \max_{B\subseteq\{e\}}\omega(D^{\sigma|B})=1
\]
for \(\sigma=\times\) and for \(\sigma=\times\ast\), and in both cases the maximum is
attained only when \(B=\{e\}\). However,
\[
        \{e\}\notin\mathcal F(D).
\]
Thus the feasible-set attainment property fails for both \(\times\) and
\(\times\ast\), even for vf-safe delta-matroids.
\end{example}

\section{Monotone maximum width sequences for non-twist partial-twuality operations}
\label{sec:monotone-star-cross-cross-star}

\begin{lemma}[\cite{BoninChunNoble2021}] 
\label{lem:sandwich} 
Let \(D=(E,\mathcal F)\) be a delta-matroid. For every feasible set \(F\in\mathcal F\), there exist \( F_{\min}\in\mathcal F_{\min}(D), \ F_{\max}\in\mathcal F_{\max}(D), \) 
such that \[ F_{\min}\subseteq F\subseteq F_{\max}. \] \end{lemma} 

\begin{lemma}[\cite{BrijderHoogeboom2011}]
\label{lem:loop-parity}
Let \(D=(E,\mathcal F)\) be a proper set system, and let \(A,Z\subseteq E\). Then
\(Z\in\mathcal F(D^{\times|A})\)
if and only if
\[\left| \{X\in\mathcal F\mid Z\setminus A\subseteq X\subseteq Z\} \right|
\]
is odd.
\end{lemma}

Let \(D\) be a proper set system, let \(\sigma\in\{\ast\times,\times\ast\}\), and let \(e\in E(D)\). We call \(e\) \(\sigma\)-decreasing in \(D\) if \[ \omega(D^{\sigma|e})<\omega(D). \] 

\begin{lemma}
\label{lem:decreasing-type}
Let \(D=(E,\mathcal F)\) be a proper set system, let
\(\sigma\in\{\ast\times,\times\ast\}\), and let \(e\in E\). If \(e\) is
\(\sigma\)-decreasing in \(D\), then \(e\) belongs to no minimum feasible set of
\(D\), and belongs to every maximum feasible set of \(D\).
\end{lemma}

\begin{proof}
By Table~\ref{tab:widths}, the width decreases under \(\ast\times\) precisely for
types \(\{uu,\ ut, tt\},\)
and decreases under \(\times\ast\) precisely for types \(\{uu, tu, tt\}.
\) Hence, if \(e\) is \(\sigma\)-decreasing for some
\(\sigma\in\{\ast\times,\times\ast\}\), then the primal type of \(e\) is either
\(u\) or \(t\), and the dual type of \(e\) is either \(u\) or \(t\).

By Proposition~\ref{lemma6}, primal type \(u\) or \(t\) means that \(e\) belongs to
no minimum feasible set of \(D\), while dual type \(u\) or \(t\) means that \(e\)
belongs to every maximum feasible set of \(D\). This proves the lemma.
\end{proof}

\begin{lemma} 
\label{lem:bad-set} 
Let \(D=(E,\mathcal F)\) be a delta-matroid, and let \( \sigma\in\{\ast\times,\times\ast\}. \) Let \(A\subseteq E\). Suppose every element of \(A\) is \(\sigma\)-decreasing in \(D\). Then \[ \omega(D^{\sigma|A})\leq \omega(D). \] 
\end{lemma}

\begin{proof} 
By Lemma~\ref{lem:decreasing-type}, for every
\(F_{\min}\in\mathcal F_{\min}(D)\) and every
\(F_{\max}\in\mathcal F_{\max}(D)\), we have
\(F_{\min}\cap A=\emptyset\) and  \(A\subseteq F_{\max}\).
Let \(Z\in\mathcal F(D^{\sigma|A})\). We distinguish two cases.

If \(\sigma=\ast\times\), then \[ D^{\ast\times|A}=(D^{\ast|A})^{\times|A}. \] By Lemma~\ref{lem:loop-parity}, applied to \(D^{\ast|A}\), there exists \(X\in\mathcal F(D^{\ast|A})\) such that \[ Z\setminus A\subseteq X\subseteq Z. \] 
Since \(X\in\mathcal F(D^{\ast|A})\),  there exists
\(F\in\mathcal F(D)\) such that \(X=F\triangle A.\)

If \(\sigma=\times\ast\), then
\[
        D^{\times\ast|A}=(D^{\times|A})^{\ast|A}.
\]
Since \(Z\in\mathcal F(D^{\times\ast|A})\), we have \(Z\triangle A\in \mathcal F(D^{\times|A})\). Set \(Y=Z\triangle A\).
By Lemma~\ref{lem:loop-parity}, there exists \(F\in\mathcal F(D)\) such that
\[
        Y\setminus A\subseteq F\subseteq Y.
\]

In either case, we have obtained a feasible set \(F\in\mathcal F(D)\).
 By Lemma~\ref{lem:sandwich}, there exist \(F_{\min}\in\mathcal F_{\min}(D)\) and  \(F_{\max}\in\mathcal F_{\max}(D)\) such that \[ F_{\min}\subseteq F\subseteq F_{\max}. \] It remains to prove that
\[F_{\min}\subseteq Z\subseteq F_{\max}.\] 

Suppose first that \(\sigma=\ast\times\). Since \(F_{\min}\cap A=\emptyset\), we have \[ F_{\min}\subseteq F\setminus A\subseteq F\triangle A=X\subseteq Z. \] Also, since \(Z\setminus A\subseteq X\subseteq Z\), we have \(Z\setminus X\subseteq A\), and hence \[ Z\subseteq X\cup A=(F\triangle A)\cup A=F\cup A\subseteq F_{\max}. \] 

Suppose next that \(\sigma=\times\ast\). Since \(F_{\min}\cap A=\emptyset\) and \(F_{\min}\subseteq F\subseteq Y\), we have \[ F_{\min}\subseteq Y\setminus A=(Z\triangle A)\setminus A=Z\setminus A \subseteq Z. \] For the upper bound, from \(Y\setminus A\subseteq F\subseteq Y\), we have \(Y\setminus F\subseteq A\), and hence \[ Y\subseteq F\cup A\subseteq F_{\max}. \] Since \(Z=Y\triangle A\subseteq Y\cup A\), it follows that \( Z\subseteq F_{\max}. \) 

Thus, in both cases, \[ F_{\min}\subseteq Z\subseteq F_{\max}. \] Therefore \[ r(D_{\min})\leq |Z|\leq r(D_{\max}). \] Since \(Z\in\mathcal F(D^{\sigma|A})\) was arbitrary, every feasible set of \(D^{\sigma|A}\) has size between \(r(D_{\min})\) and \(r(D_{\max})\). Hence \[ \omega(D^{\sigma|A})\leq r(D_{\max})-r(D_{\min})=\omega(D). \]
\end{proof}

\begin{lemma}
\label{lem:accessibility} 

Let \(D=(E,\mathcal F)\) be a delta-matroid, let $\sigma\in\{\ast\times,\times\ast\}$, and let \(S\subseteq E\). Choose \(A\subseteq S\) such that
\[
        \omega(D^{\sigma|A})
        =
        \max_{B\subseteq S}\omega(D^{\sigma|B}),
\]
and such that no proper subset of \(A\) satisfies this equality.
If \(A\neq\emptyset\), then there exists \(e\in A\) such that \[\omega(D^{\sigma|e})\geq \omega(D).\] 
\end{lemma} 

\begin{proof} 
Suppose, for contradiction, that $\omega(D^{\sigma|e})<\omega(D)$ for every \(e\in A\). Then every element of \(A\) is \(\sigma\)-decreasing in \(D\). By Lemma~\ref{lem:bad-set}, we have \[ \omega(D^{\sigma|A})\leq \omega(D). \] Since \(\emptyset\subseteq S\) and \(D^{\sigma|\emptyset}=D\), we have \[ \omega(D) = \omega(D^{\sigma|\emptyset}) \leq \max_{B\subseteq S}\omega(D^{\sigma|B}) = \omega(D^{\sigma|A}). \] Thus \[ \omega(D^{\sigma|A})=\omega(D). \] Hence, $\emptyset $ also attains the maximum over all subsets of \(S\). Since \(\emptyset\subsetneq A\), this contradicts the choice of \(A\). Therefore there exists \(e\in A\) such that \[ \omega(D^{\sigma|e})\geq \omega(D). \] 
\end{proof} 

\begin{theorem}
\label{thm:restricted} 

Let \(D=(E,\mathcal F)\) be a vf-safe delta-matroid, let $\sigma\in\{\ast\times,\times\ast\},$  and let  \(S\subseteq E\). Then there exists \(A\subseteq S\), together with an ordering
\(e_1,\ldots,e_k\) of the elements of \(A\),  such that \[ \omega(D^{\sigma|A}) = \max_{B\subseteq S}\omega(D^{\sigma|B}), \] and the sequence \[ \omega(D),\, \omega(D^{\sigma|e_1}),\, \omega(D^{\sigma|\{e_1,e_2\}}),\ldots,\, \omega(D^{\sigma|A}) \] is non-decreasing. 
\end{theorem} 

\begin{proof}
We prove the statement by induction on \(|S|\). The induction hypothesis is assumed for all vf-safe delta-matroids and all subsets of size smaller than \(|S|\).

If \(S=\emptyset\), then the only possible choice is \(A=\emptyset\), and the conclusion is immediate. Now assume \(S\neq\emptyset\), and assume the theorem is known for all subsets
of smaller cardinality.
Choose \(A\subseteq S\) such that
\[
        \omega(D^{\sigma|A})
        =
        \max_{B\subseteq S}\omega(D^{\sigma|B}),
\]
and such that no proper subset of \(A\) satisfies this equality.
If \(A=\emptyset\), then the empty sequence proves the theorem. Suppose \(A\neq\emptyset\). By Lemma~\ref{lem:accessibility}, there exists
\(e\in A\) such that
\[
        \omega(D^{\sigma|e})\geq \omega(D).
\]
Set \(D_1=D^{\sigma|e}\) and \(S_1=S\setminus\{e\}.
\)
Since \(D\) is vf-safe and \(\sigma\) is a word in twists and loop complementations,
\(D_1\) is also vf-safe.

Let \(A_1=A\setminus\{e\}.\)
We claim that \(A_1\) attains the maximum of \(\omega(D_1^{\sigma|B})\) over all
\(B\subseteq S_1\). Indeed, for every \(B\subseteq S_1\), operations on distinct elements commute, so
\[
        D_1^{\sigma|B}
        =
        (D^{\sigma|e})^{\sigma|B}
        =
        D^{\sigma|(\{e\}\cup B)}.
\]
Since \(\{e\}\cup B\subseteq S\), we have
\[
        \omega(D_1^{\sigma|B})
        =
        \omega(D^{\sigma|(\{e\}\cup B)})
        \leq
        \max_{C\subseteq S}\omega(D^{\sigma|C})
        =
        \omega(D^{\sigma|A}).
\]
Hence
\[
        \max_{B\subseteq S_1}\omega(D_1^{\sigma|B})
        \leq
        \omega(D^{\sigma|A}).
\]
On the other hand,
\[
        D_1^{\sigma|A_1}
        =
        (D^{\sigma|e})^{\sigma|A_1}
        =
        D^{\sigma|A}.
\]
Therefore
\[
        \omega(D_1^{\sigma|A_1})
        =
        \omega(D^{\sigma|A}).
\]
Since \(A_1\subseteq S_1\), we also have
\[
        \omega(D_1^{\sigma|A_1})
        \leq
        \max_{B\subseteq S_1}\omega(D_1^{\sigma|B}).
\]
Thus
\[
        \omega(D_1^{\sigma|A_1})
        =
        \max_{B\subseteq S_1}\omega(D_1^{\sigma|B})
        =
        \omega(D^{\sigma|A}).
\]

Now apply the induction hypothesis to the vf-safe delta-matroid \(D_1\) and the
subset \(S_1\). There exists a subset \(C\subseteq S_1\), with an ordering
\[
        C=\{f_1,\ldots,f_m\},
\]
such that
\[
        \omega(D_1^{\sigma|C})
        =
        \max_{B\subseteq S_1}\omega(D_1^{\sigma|B})
        =
        \omega(D^{\sigma|A}),
\]
and the sequence
\[
        \omega(D_1),\,
        \omega(D_1^{\sigma|f_1}),\,
        \omega(D_1^{\sigma|\{f_1,f_2\}}),\ldots,\,
        \omega(D_1^{\sigma|C})
\]
is non-decreasing. Since
\[
        \omega(D_1)=\omega(D^{\sigma|e})\geq \omega(D),
\]
the sequence
\(\{e,\ f_1,\ldots,f_m\}\)
satisfies
\[
        \omega(D)
        \leq
        \omega(D^{\sigma|e})
        =
        \omega(D_1)
        \leq
        \omega(D_1^{\sigma|f_1})
        \leq
        \cdots
        \leq
        \omega(D_1^{\sigma|C}).
\]
Since 
\( D_1^{\sigma|C}=D^{\sigma|(\{e\}\cup C)},\)
we have
\[
        \omega(D^{\sigma|(\{e\}\cup C)})
        =
        \max_{B\subseteq S}\omega(D^{\sigma|B}).
\]
Therefore the subset
\(\{e\}\cup C\subseteq S\)
with the ordering \(\{e,f_1,\ldots,f_m\}\) has the required properties. This completes the induction.
\end{proof}

\begin{corollary}
\label{cor:star-cross-cross-star-global}
Let \(D=(E,\mathcal F)\) be a vf-safe delta-matroid, and let
\(\sigma\in\{\ast\times,\times\ast\}\). Then there exists a subset \(A\subseteq E\), together with an ordering
\(e_1,\ldots,e_k\) of the elements of \(A\), such that
\[
        \omega(D^{\sigma|A})
        =
        \max_{B\subseteq E}\omega(D^{\sigma|B}),
\]
and the sequence
\[
        \omega(D),\,
        \omega(D^{\sigma|e_1}),\,
        \omega(D^{\sigma|\{e_1,e_2\}}),\ldots,\,
        \omega(D^{\sigma|A})
\]
is non-decreasing.
\end{corollary} 

\begin{proof}
Apply Theorem~\ref{thm:restricted} with \(S=E\).
\end{proof}

We have so far discussed the partial-\(\ast\times\) and partial-\(\times\ast\)
operations. For completeness, we briefly record the corresponding monotone
statements for the remaining two operations \(\times\) and \(\ast\times\ast\).
Recall that
\[
        \ast\times\ast=\times\ast\times
\]
in the \(S_3\) action generated by twist and loop complementation.

\begin{lemma}
\label{lem:bad-set-loop}
Let \(D=(E,\mathcal F)\) be a vf-safe delta-matroid, and let \(A\subseteq E\).
Suppose that for every \(e\in A\), \(\omega(D^{\times|e})<\omega(D).\)
Then
\[
        \omega(D^{\times|A})\leq \omega(D).
\]
\end{lemma}

\begin{proof}
Set \( r=r(D_{\min})\) and 
\(R=r(D_{\max})\). By Table~\ref{tab:widths}, the width decreases under \(\times|e\) precisely for
elements of types
\(\{pt, ut, tt\}.\)
In all these cases, the dual type of \(e\) is \(t\). By Proposition \ref{lemma6}, every maximum feasible set of \(D\) contains \(e\). Hence, for every \(F_{\max}\in\mathcal F_{\max}(D)\),  \(A\subseteq F_{\max}.\)

Let \(Z\in\mathcal F(D^{\times|A}).\)
By Lemma \ref{lem:loop-parity}, there exists
\(X\in\mathcal F(D)\)
such that
\[
        Z\setminus A\subseteq X\subseteq Z.
\]
Hence,
\[
        |Z|\geq |X|\geq r.
\]
By Lemma~\ref{lem:sandwich}, 
there exists \(F_{\max}\in\mathcal F_{\max}(D)\) such that
\(X\subseteq F_{\max}.\)
Since \(A\subseteq F_{\max}\) and \(Z\setminus A\subseteq X\subseteq Z\), we have
\(Z\setminus X\subseteq A.\)
Therefore,
\[
        Z\subseteq X\cup A\subseteq F_{\max}.
\]
Hence,
\[
        |Z|\leq |F_{\max}|=R.
\]
Thus every feasible set \(Z\) of \(D^{\times|A}\) satisfies
\[
        r\leq |Z|\leq R.
\]
Consequently,
\[
        \omega(D^{\times|A})\leq R-r=\omega(D).
\]
\end{proof}

The same accessibility and induction argument used in
Theorem~\ref{thm:restricted} now gives the following result.

\begin{theorem}
\label{thm:loop-monotone}
Let \(D=(E,\mathcal F)\) be a vf-safe delta-matroid, and let \(S\subseteq E\).
Then there exists \(A\subseteq S\), with an ordering \(e_1,\ldots,e_k\) of the elements of \(A\)
such that
\[
        \omega(D^{\times|A})
        =
        \max_{B\subseteq S}\omega(D^{\times|B}),
\]
and the sequence
\[
        \omega(D),\,
        \omega(D^{\times|e_1}),\,
        \omega(D^{\times|\{e_1,e_2\}}),\ldots,\,
        \omega(D^{\times|A})
\]
is non-decreasing.
\end{theorem}

\begin{proof}
The proof is identical to the proof of Theorem~\ref{thm:restricted}, with
Lemma~\ref{lem:bad-set} replaced by Lemma~\ref{lem:bad-set-loop}.
\end{proof}

The operation \(\ast\times\ast\) follows from the \(\times\)-case by duality.

\begin{lemma}
\label{lem:dual-loop-dualpivot}
For every set system \(D=(E,\mathcal F)\) and every \(A\subseteq E\),
\[
        D^{\ast\times\ast|A}
        =
        \bigl((D^\ast)^{\times|A}\bigr)^\ast.
\]
Consequently,
\[
        \omega(D^{\ast\times\ast|A})
        =\omega((D^\ast)^{\times|A}).
\]
\end{lemma}

\begin{proof}
Since \(D^\ast=D^{\ast|E}\), we have
\[
        \bigl((D^\ast)^{\times|A}\bigr)^\ast
        =
        \bigl((D^{\ast|E})^{\times|A}\bigr)^{\ast|E}.
\]
On each element of \(A\), the operations applied are
\(\ast,\ \times,\ \ast,\)
and hence give \(\ast\times\ast\). On each element of \(E\setminus A\), the two
twists cancel. Therefore
\[
        \bigl((D^\ast)^{\times|A}\bigr)^\ast
        =
        D^{\ast\times\ast|A}.
\]
The equality of widths follows since duality preserves width.
\end{proof}

\begin{theorem}
\label{cor:dual-pivot-monotone}
Let \(D=(E,\mathcal F)\) be a vf-safe delta-matroid, and let \(S\subseteq E\).
Then there exists \(A\subseteq S\), with an ordering \(e_1,\ldots,e_k\) of the elements of \(A\)
such that
\[
        \omega(D^{\ast\times\ast|A})
        =
        \max_{B\subseteq S}\omega(D^{\ast\times\ast|B}),
\]
and the sequence
\[
        \omega(D),\,
        \omega(D^{\ast\times\ast|e_1}),\,
        \omega(D^{\ast\times\ast|\{e_1,e_2\}}),\ldots,\,
        \omega(D^{\ast\times\ast|A})
\]
is non-decreasing.
\end{theorem}

\begin{proof}
Since \(D\) is vf-safe, its dual \(D^\ast\) is also vf-safe. Apply
Theorem~\ref{thm:loop-monotone} to \(D^\ast\) and \(S\). Then use
Lemma~\ref{lem:dual-loop-dualpivot} and the fact that duality preserves width.
\end{proof}

\section{Consequences for ribbon graphs} \label{sec:ribbon-graphs} 

In this section we translate our delta-matroid results into ribbon graphs. For background on ribbon graphs, partial duality, and Petriality, we refer the reader to \cite{Ellis-Monaghan2013}.

Let \(G\) be a ribbon graph with edge set \(E(G)\). Recall that the delta-matroid associated with \(G\) is \( D(G)=(E(G),\mathcal F(G)), \) where \[ \mathcal F(G) = \{F\subseteq E(G)\mid F \text{ is the edge set of a spanning quasi-tree of }G\}. \] Thus feasible sets of \(D(G)\) are precisely edge sets of spanning quasi-trees of \(G\) \cite{ChunMoffattNobleRueckriemen2019PLMS}. We use the standard ribbon-group notation: \(\delta\) denotes partial duality and \(\tau\) denotes partial Petriality. Ribbon-graphic delta-matroids are vf-safe \cite{BoninChunNoble2021}. Moreover, for every \(A\subseteq E(G)\), partial duality corresponds to twist, and partial Petriality corresponds to loop complementation \cite{ChunMoffattNobleRueckriemen2019PLMS}: \[ D(G^{\delta|A})=D(G)^{\ast|A}, \qquad D(G^{\tau|A})=D(G)^{\times|A}. \] We use the convention that the word order is preserved, that is, \[ G^{\delta\tau|A}=(G^{\delta|A})^{\tau|A}, \qquad G^{\tau\delta|A}=(G^{\tau|A})^{\delta|A}. \] Consequently, \[ D(G^{\delta\tau|A}) = D(G)^{\ast\times|A}, \qquad D(G^{\tau\delta|A}) = D(G)^{\times\ast|A}. \] 

We also use the standard identity that the width of the ribbon-graphic delta-matroid is equal to the Euler genus of the ribbon graph \cite{ChunMoffattNobleRueckriemen2019PLMS}: \[ \omega(D(G))=\varepsilon(G). \] Therefore, for every \(A\subseteq E(G)\), \[ \omega(D(G)^{\ast\times|A}) = \varepsilon(G^{\delta\tau|A}), \] and \[ \omega(D(G)^{\times\ast|A}) = \varepsilon(G^{\tau\delta|A}). \] For a ribbon graph \(G\), and for
\(\rho\in\{\delta,\tau,\delta\tau,\tau\delta,\delta\tau\delta\},\)
define the maximum partial-\(\rho\) Euler genus of \(G\) by
\[
        \varepsilon_M^{\rho}(G)
        =
        \max_{A\subseteq E(G)}
        \varepsilon(G^{\rho|A}).
\]

\begin{corollary}
\label{cor:ribbon-twuality-monotone}
Let \(G\) be a ribbon graph, and let
\(\rho\in\{\delta\tau,\tau\delta,\tau,\delta\tau\delta\}.\)
Then there exists a subset \(A\subseteq E(G)\), together with an ordering
\(e_1,\ldots,e_k\) of the elements of \(A\), such that
\[
        \varepsilon(G^{\rho|A})
        =
        \varepsilon_M^{\rho}(G),
\]
and the sequence
\[
        \varepsilon(G),\,
        \varepsilon(G^{\rho|\{e_1\}}),\,
        \varepsilon(G^{\rho|\{e_1,e_2\}}),\ldots,\,
        \varepsilon(G^{\rho|A})
\]
is non-decreasing.
\end{corollary}

\begin{proof}
Define
\[
        \sigma=
        \begin{cases}
        \ast\times, & \text{if } \rho=\delta\tau,\\
        \times\ast, & \text{if } \rho=\tau\delta,\\
        \times, & \text{if } \rho=\tau,\\
        \ast\times\ast, & \text{if } \rho=\delta\tau\delta.
        \end{cases}
\]
Since ribbon-graphic delta-matroids are vf-safe, \(D(G)\) is vf-safe. If
\(\rho\in\{\delta\tau,\tau\delta\}\), apply Theorem~\ref{thm:restricted} to
\(D(G)\) with \(S=E(G)\). If \(\rho=\tau\), apply
Theorem~\ref{thm:loop-monotone}. If \(\rho=\delta\tau\delta\), apply
Theorem~\ref{cor:dual-pivot-monotone}. In each case, we obtain a subset
\(A\subseteq E(G)\), together with an ordering \(e_1,\ldots,e_k\) of the elements
of \(A\), such that
\[
        \omega(D(G)^{\sigma|A})
        =
        \max_{B\subseteq E(G)}
        \omega(D(G)^{\sigma|B}),
\]
and the corresponding width sequence is non-decreasing.

By the correspondence between ribbon-graph twualities and delta-matroid operations,
\[
        D(G^{\rho|B})=D(G)^{\sigma|B}
\]
for every \(B\subseteq E(G)\). Moreover,
\[
        \omega(D(H))=\varepsilon(H)
\]
for every ribbon graph \(H\). Therefore
\[
        \varepsilon(G^{\rho|A})
        =
        \max_{B\subseteq E(G)}
        \varepsilon(G^{\rho|B})
        =
        \varepsilon_M^{\rho}(G).
\]
The non-decreasing width sequence translates exactly into the stated
non-decreasing Euler-genus sequence.
\end{proof}

We also record the following consequence of the feasible-set attainment results
for the partial-\(\ast\times\) and partial-\(\ast\times\ast\) operations.

\begin{corollary}
\label{cor:ribbon-feasible-quasitree}
Let \(G\) be a ribbon graph, and let
\(\rho\in\{\delta\tau,\delta\tau\delta\}.\)
Then there exists a spanning quasi-tree \(Q\) of \(G\) such that, writing
\(E(Q)\) for its edge set,
\[
        \varepsilon(G^{\rho|E(Q)})
        =
        \varepsilon_M^{\rho}(G).
\]
\end{corollary}

\begin{proof}
Define
\[
        \sigma=
        \begin{cases}
        \ast\times, & \text{if } \rho=\delta\tau,\\
        \ast\times\ast, & \text{if } \rho=\delta\tau\delta.
        \end{cases}
\]
If \(\rho=\delta\tau\), apply Theorem~\ref{thm:star-cross-feasible-max} to
\(D(G)\). If \(\rho=\delta\tau\delta\), apply
Theorem~\ref{thm:star-cross-star-feasible-max} to \(D(G)\). In either case,
there exists a feasible set \(F\in\mathcal F(D(G))\) such that
\[
        \omega(D(G)^{\sigma|F})
        =
        \max_{B\subseteq E(G)}
        \omega(D(G)^{\sigma|B}).
\]
Since \(F\in\mathcal F(D(G))\), there is a spanning quasi-tree \(Q\) of \(G\)
such that  \(F=E(Q).\)
By the correspondence between ribbon-graph twualities and delta-matroid
operations,
\[
        D(G^{\rho|B})=D(G)^{\sigma|B}
\]
for every \(B\subseteq E(G)\). Using also
\[
        \omega(D(H))=\varepsilon(H)
\]
for every ribbon graph \(H\), we obtain
\[
        \varepsilon(G^{\rho|E(Q)})
        =
        \max_{B\subseteq E(G)}
        \varepsilon(G^{\rho|B})
        =
        \varepsilon_M^{\rho}(G).
\]
\end{proof}

\section*{Acknowledgements} 
This work is supported by NSFC (No. 12471326).

 \end{document}